\pdfoutput=1
\documentclass[a4paper,11pt]{article}
\usepackage{hyperref}
\hypersetup{
    colorlinks=true,
    linktoc=all,
    linkcolor=red,     %choose some color if you want links to stand out
    citecolor=blue,     %choose some color if you want links to stand out
}
\usepackage{graphicx}
\usepackage{amsmath,amssymb,amsfonts,amsthm,graphics,enumitem}
\usepackage{latexsym, bm}
\usepackage{multicol}
\usepackage{comment}
\usepackage{indentfirst}

\usepackage{setspace}
\newtheorem{theorem}{Theorem}[section]
\newtheorem{lemma}{Lemma}[section]
\newtheorem{claim}{Claim}

\newtheorem{proposition}{Proposition}[section]
\newtheorem{definition}{Definition}[section]

\newtheorem{conjecture}{Conjecture}[section]

\newcommand{\ignore}[1]{}

\begin{document}
\begin{spacing}{1}
\date{}
\title{Connectivity keeping pendant extensions of paths in $k$-connected graphs and triangle-free graphs\thanks{The work is supported by NSFC(11871015) and CGLSTDFFJ(2023L3003). E-mail: menghanma664@163.com (M. Ma), qliu@fzu.edu.cn (Q. Liu), 1519036064@qq.com (L. Zhang), yhong@fzu.edu.cn (Y. Hong).}}

\author{Menghan Ma$^a$, Qinghai Liu$^{a,c}$\footnote{Corresponding author.}, Liping Zhang$^{a}$, Yanmei Hong$^{b,c}$\\[1mm]
{\small \it
$^a$Center for Discrete Mathematics, Fuzhou University, Fuzhou, Fujian, 350108 China}\\
%{\small \it
%$^b$Fujian Science \& Technology Innovation Laboratory for Optoelectronic Information of China, Fuzhou, Fujian, 350108 China}\\
{\small \it
$^b$School of Mathematics and Statistics, Fuzhou University, Fuzhou, Fujian, 350108 China}\\
{\small \it
$^c$Key Laboratory for Operations Research and Cybernetics of Fujian Universities}
}
\begin{comment}
\title{Connectivity keeping pendant extensions of paths in $k$-connected graphs and triangle-free graphs}
\date{}

\author{
xx\footnote{Center for Discrete Mathematics, Fuzhou University,
Fuzhou, 350108, P.~R.~China. Email: {\tt xx}. } \;\;   \;\; xx\footnote{Center for Discrete Mathematics, Fuzhou University,
Fuzhou, 350108, P.~R.~China. Email: {\tt xx}. } \;\; xx\footnote{Center for Discrete Mathematics, Fuzhou University,
Fuzhou, 350108, P.~R.~China. Email: {\tt xx}. }
}
\end{comment}
\maketitle
\begin{abstract}
Motivated by Mader's conjecture on connectivity keeping trees, we study trees obtained from paths by adding one pendant vertex, as well as related problems in triangle-free graphs.

For an integer $m$ and $1\leq i\leq m-1$, let $P_m^+(i)$ denote the tree obtained from a path of order $m-1$ by adding one pendant vertex adjacent
to its $i$th vertex. We prove that, for positive integers $k,m,1\leq i\leq m-1$, every $k$-connected graph $G$ with $\delta(G)\geq \lfloor \frac{3k}{2}\rfloor+m-1$ contains a subgraph $T\cong P_m^+(i)$ such that $\kappa(G-V(T))\geq k$. This confirms Mader's conjecture for all pendant extensions of paths.

For highly connected triangle-free graphs, a connectivity keeping result for paths was obtained in [J. Combin. Theory Ser. B, 174 (2025), 190-206]. Let $(X,Y)$ be the bipartition of $P_m^+(i)$. We further prove that every $k$-connected triangle-free graph $G$ with $\delta(G)\geq k+\max\{|X|,|Y|\}+[P_m^+(i)\text{ is bad}]$ contains a subgraph $T\cong P_m^+(i)$ such that $\kappa(G-V(T))\geq k$, where we use Iverson's convention for $[P_m^+(i)\text{ is bad}]$. This extends the corresponding result for paths to pendant extensions of paths.
\medskip


\end{abstract}

\section{Introduction}

In this paper, all graphs are considered to be finite, undirected, and simple. We refer to the book \cite{Bondy} for any undefined notation and terminology in the following.

For a graph $G$, we denote its {\it vertex set}, {\it edge set}, {\it connectivity}, and {\it minimum degree} as $V(G)$, $E(G)$, $\kappa(G)$, and $\delta(G)$, respectively. For a vertex set $U\subseteq V(G)$, $G[U]$ denotes the subgraph induced by $U$ and $G-U=G[V(G)\setminus U]$.
The {\it connectivity} $\kappa(G)$ of $G$ is the cardinality of the minimum vertex subset $S\subseteq V(G)$ such that $G-S$ is disconnected or trivial. A graph $G$ is {\it $k$-connected} if $\kappa(G)\geq k$.

In 1972, Chartrand, Kaigars and Lick \cite{Chartrand} first proved that every $k$-connected graph $G$ with  $\delta(G)\ge\left\lfloor\frac{3k}{2}\right\rfloor$ contains a vertex $u$ such that $\kappa(G-u)\geq k$. In 2008, Fujita and Kawarabayashi \cite{Fujita} proved that every $k$-connected graph $G$ with $\delta(G)\ge\left\lfloor\frac{3k}{2}\right\rfloor+2$ contains an edge $uv$ such that $\kappa(G-\{u,v\})\geq k$, and they proposed a conjecture about a redundant subgraph in a $k$-connected graph.

\begin{conjecture}[\upshape Fujita and Kawarabayashi~\cite{Fujita}]\label{Fujita} For all positive integers $k,m$, there is a (least) non-negative integer $f_k(m)$ such that every $k$-connected graph $G$ with $\delta(G)\ge\left\lfloor\frac{3k}{2}\right\rfloor+f_k(m)-1$ contains a connected subgraph $W$ of order $m$ such that $\kappa(G-V(W))\geq k$.
\end{conjecture}

In 2010, Mader \cite{Mader1} proved the validity of Conjecture \ref{Fujita}, explicitly established that $f_k(m)=m$, and showed that the subgraph $W$ can be chosen as a path.

\begin{theorem}[\upshape Mader~\cite{Mader1}]\label{mader1}
For all positive integers $k,m$, every $k$-connected graph $G$ with $\delta(G)\ge\left\lfloor\frac{3k}{2}\right\rfloor+m-1$ contains a path $P$ of order $m$ such that $\kappa(G-V(P))\geq k$.
\end{theorem}

Mader further conjectured that the path in the theorem above can be replaced by any tree of the same order.

\begin{conjecture}[\upshape Mader~\cite{Mader1}]\label{Mad}
For any tree $T$ of order $m$ and any integer $k$, every $k$-connected graph $G$ with $\delta(G)\ge\left\lfloor\frac{3k}{2}\right\rfloor+m-1$ contains a subtree $T'\cong T$ such that $\kappa(G-V(T'))\geq k$.
\end{conjecture}

As early as 2009, the result published by Diwan and Tholiya \cite{Diwan} established the validity of Conjecture \ref{Mad} for the case when $k=1$. For the case when $k=2$, there have been some studies on special trees \cite{Hasunuma2, HongLLY, Lu, Luo, Tian1, Tian2}. In \cite{HongLiu}, Hong and Liu proved that the conjecture holds for any tree when $k\leq 3$. For general $k$, Mader \cite{Mader2} proved the conjecture holds if $\delta(G)\geq 2(k+m-1)^2+m-1$. Recently, Liu, Ying and Hong \cite{LYH} optimized the minimum degree condition to be linear and proved that the conjecture holds if $\delta(G)\geq 3k+4m-6$. Nevertheless, the sharp bound predicted by Mader's conjecture remains open for many natural classes of non-path trees when $k\geq 4$. For more studies on Mader's conjecture, one can refer to the survey \cite{Tian3}.

In this paper, we study trees obtained from paths by adding a single pendant vertex, a natural and fundamental family of non-path trees. For integers $m$ and $i$ with $1\leq i\leq m-1$, let $P_m^+(i)$ denote the tree obtained from a path $v_1v_2\cdots v_{m-1}$ by adding a new vertex $u$ adjacent to $v_i$. Thus $P_m^+(i)$ has order $m$ and may be viewed as a pendant extension of a path. Observe that $P_m^+(1)$ and $P_m^+(m-1)$ are both paths of order $m$, while $P_4^+(2)\cong K_{1,3}$.

Our first main result confirms Mader's conjecture for all pendant extensions of paths.
\begin{comment}
\begin{theorem}\label{K13}
Let $k$ be a positive integer. Every $k$-connected graph $G$ with $\delta(G)\ge\left\lfloor\frac{3k}{2}\right\rfloor+3$ contains a subgraph $T\cong K_{1,3}$ such that $\kappa(G-V(T))\geq k$. Moreover, if $\kappa(G)=k$, then $T$ can be chosen so that $T\subseteq F$, where $F$ is an arbitrary end of $G$.
\end{theorem}
\end{comment}

\begin{theorem}\label{1tm-t-2}
Let $k,m,i$ be positive integers with $1\leq i\leq m-1$. Every $k$-connected graph $G$ with $\delta(G)\ge\left\lfloor\frac{3k}{2}\right\rfloor+m-1$ contains a subgraph $T\cong P_m^+(i)$ such that $\kappa(G-V(T))\geq k$. Moreover, if $\kappa(G)=k$, then $T$ can be chosen so that $T\subseteq F$, where $F$ is an arbitrary end of $G$.
\end{theorem}

We also consider the corresponding problem for triangle-free graphs. Chu, Fujita, Park and Ryu \cite{CFPR} proved that for any tree $T$ of order $m$, every $k$-connected triangle-free graph $G$ with $\delta(G)\geq 2k+3m-4$ contains a subtree $T'\cong T$ such that $\kappa(G-V(T'))\geq k$. For paths, Fujita \cite{Ftriangle} established the following two results.

\begin{theorem}[\upshape Fujita~\cite{Ftriangle}]\label{F1}
Let $k,m$ be positive integers. If $G$ is a $k$-connected triangle-free graph with $\delta(G)\geq k+\frac{m+1}{2}$, then $G$ contains a path $P$ on
$m$ vertices such that $G-V(P)$ is $k$-connected. Moreover, $P$ has the following additional properties according to the value of $\kappa(G)$.
\begin{enumerate}
    \item If $\kappa(G)=k$, then for every end $F$ of $G$ and every
vertex $v\in V(F)$, there exists such a path $P$ with
$V(P)\subseteq V(F)$ and with $v$ as an endvertex.

    \item If $\kappa(G)\geq k+1$, then we can choose any vertex of
    $G$ as an end vertex of $P$.
\end{enumerate}
\end{theorem}
\begin{theorem}[\upshape Fujita~\cite{Ftriangle}]\label{F2}
For an integer $k\geq 3$, if $G$ is a $k$-connected triangle-free graph with $\delta(G)\geq k+1$, then $G$ contains an edge $e=xy$ such that $\kappa(G-\{x,y\})\geq k$.
\end{theorem}

Although Theorem~\ref{F2} is stated for $k\geq 3$, the same proof also gives the corresponding case $k=2$. For completeness, we state this case separately.

\begin{proposition}\label{propk=2}
If $G$ is a $2$-connected triangle-free graph with $\delta(G)\geq 3$, then $G$ contains an edge $e=xy$ such that $\kappa(G-\{x,y\})\geq 2$.
\end{proposition}

Theorem \ref{mader1} already covers the case of trees of order two when $k=1$, and the results above settle the corresponding triangle-free cases for trees of order at most three. Therefore, the first nontrivial cases for pendant extensions of paths arise when $m\geq 4$, which is precisely the range considered in this paper.

We now introduce the definition of good and bad pendant extensions for our triangle-free result. Let $P_m^+(i)$ be obtained from the path
$v_1v_2\cdots v_{m-1}$ by adding a new vertex $u$ adjacent to $v_i$. Viewing $v_i$ as the {\it center} of $P_m^+(i)$, the paths $v_{i-1}\cdots v_1$, $v_{i+1}\cdots v_{m-1}$, and $u$ are called the three {\it legs} of $P_m^+(i)$. We say that $P_m^+(i)$ is good if, after deleting $u$ together
with its neighbor $v_i$, the remaining graph has a connected component that is a path of odd order at least three. Otherwise, we say that $P_m^+(i)$ is bad. Equivalently, $P_m^+(i)$ is good if at least one of $i-1$ and $m-1-i$ is odd and at least three. We use Iverson's convention: for an event $A$, $[A]=1$ if $A$ is true, and $[A]=0$ otherwise.

Our triangle-free result is as follows.

\begin{theorem}\label{triangle}
Let $k,m,i$ be positive integers with $m\geq 4$ and $1\leq i\leq m-1$, and let $(X,Y)$ be the bipartition of $P_m^+(i)$. Every $k$-connected triangle-free graph $G$ with $\delta(G)\ge k+\max\{|X|,|Y|\}+[P_m^+(i)\text{ is bad}]$ contains a subgraph $T\cong P_m^+(i)$ such that $\kappa(G-V(T))\geq k$. Moreover, if $\kappa(G)=k$, then $T$ can be chosen so that $T\subseteq F$, where $F$ is an arbitrary end of $G$.
\end{theorem}

Let $P=P_m^+(1)$ be a path of order $m\ge4$, with bipartition
$(X,Y)$. Then $\max\{|X|,|Y|\}=\lceil \frac{m}{2}\rceil$.
Deleting the added leaf together with its neighbor leaves
a single path of order $m-2$. Hence $P_m^+(1)$ is bad
exactly when $m$ is even. Since $\delta(G)$ is an integer,
the degree condition in Theorem \ref{F1} is equivalent to $\delta(G)\ge k+\left\lceil\frac{m+1}{2}\right\rceil
=k+\max\{|X|,|Y|\}+[m\text{ is even}]$. Thus Theorem \ref{triangle} extends the corresponding existence result
for paths to pendant extensions of paths.

The rest of this paper is organized as follows. In Section \ref{preliminaries}, we summarize the notation and terminology used throughout the paper. In Section \ref{genthm}, we complete the proof of Theorem \ref{1tm-t-2}. The case of $K_{1,3}$ is established in Subsection \ref{K13}, while the complete proof is given in Subsection \ref{genfull}. Section \ref{trithm} is devoted to the proof of Theorem \ref{triangle}. We first treat the case $K_{1,3}$ in Subsection \ref{tri4} and then present the full proof in Subsection \ref{trifull}. Section \ref{Prop1.1} contains the proof of Proposition \ref{propk=2}.
\section{Preliminaries}\label{preliminaries}
In this section, we introduce some notation and definitions and collect several lemmas that will be used in the proofs of our main results.

For a graph $G$ and its vertex $x$, let $d_G(x)$ and $N_G(x)$ denote
the {\it degree} and {\it neighborhood} of $x$ in $G$, respectively.
When considering a vertex subset $X\subseteq V(G)$, we define $N_G(X)=\cup_{x\in X}N_G(x)\setminus X$. For positive integers $a<b$, let $[a]=\{1,2,\cdots,a\}$ and $[a,b]=\{a,a+1,\cdots, b\}$.

A graph $G$ is called a {\it complete graph} if every pair of distinct vertices is adjacent. The complete graph with $n$ vertices is denoted by $K_n$. A bipartite graph $G$ with bipartition $(X,Y)$ is called a {\it complete bipartite graph} if every vertex of $X$ is adjacent to every vertex of $Y$. The complete bipartite graph with $|X|=m$ and $|Y|=n$ is denoted by $K_{m,n}$.
For a subgraph $H$, write $|H|=|V(H)|$ and $N_G(H)=N_G(V(H))$. Whenever a subgraph occurs in a set intersection or union, it is identified with its vertex set.

In a graph $G$, for two disjoint paths $P=v_1v_2\dots v_k$ and $Q=u_1u_2\dots u_m$ such that $v_ku_1\in E(G)$,  we use $PQ$ to denote the path $v_1v_2\dots v_ku_1u_2\dots u_m$. We simply write $vQ$ rather than  $PQ$ if $P$ only contains one vertex $v$, and write $Pu$ rather than $PQ$ if $Q$ only contains one vertex $u$. For a path $P$ and any vertices $x,y\in V(P)$, we use $P[x,y]$ to denote the subpath of $P$ from $x$ to $y$, and use $P(x,y)$ to denote the path obtained from $P[x,y]$ by deleting vertices $x$ and $y$.

A {\it vertex cut set} of a graph $G$ is a subset $S\subseteq V(G)$ such that $G-S$ is disconnected or trivial, if $|S|=k$, then we say $S$ is a {\it $k$-cutset} of $G$. Given a vertex cut set $S$, a non-empty union $F$ of components of $G-S$ with $G-(S\cup V(F))\neq \emptyset$ is called a {\it semifragment} of $G$ with respect to $S$, its complement $\bar{F}=G-(S\cup V(F))$ is the {\it complementary semifragment}. If $S$ is a minimum cut set of $G$, then $F$ is called a {\it fragment} of $G$ to $S$ if $N_G(F)=S$. A fragment that contains no other fragments of $G$ is called an {\it end} of $G$. Every non-complete graph contains at least two ends.
\begin{lemma}[\upshape Mader~\cite{Mader1}]\label{mpath}
Let $k,m$ be positive integers, let $G$ be a graph with $\kappa(G)=k$ and let $F$ be an end of $G$. If $\delta(G)\geq \left\lfloor\frac{3k}{2}\right\rfloor+m-1$, then for every
$u\in V(F)$ there exists a path $P\subseteq F$ of order $m$
starting at $u$ such that $\kappa(G-V(P))\ge k$.
\end{lemma}
\begin{lemma}\label{subend}
Let $k$ be a positive integer, let $G$ be a graph with $\kappa(G)=k$, let $S$ be a $k$-cutset of $G$ and $F$ be an end of $G$ to $S$. If $\emptyset\neq W\subsetneq F$ and $\kappa(G-V(W))\geq k$, then $\kappa(G-V(W))=k$ and there is an end $F'\subseteq F$ of $G-V(W)$ such that $N_G(W)\cap V(F')\neq \emptyset$.
\end{lemma}
\begin{proof}
Assume $F$ is an end of $G$ and $W\subsetneq F$ such that $\kappa(G-V(W))\geq k$. Since both $F-V(W)$ and $\bar F$ are nonempty, $S$ is still a $k$-cutset of $G-V(W)$. Thus, $\kappa(G-V(W))=k$ and $F-V(W)$ is a fragment of $G-V(W)$ to $S$. Let $H=G-V(W)$. Choose an end $F'$ of $H$ contained in the fragment $F-V(W)$. If $N_G(W)\cap V(F')=\emptyset$, then $N_G(F')=N_H(F')$ has size $k$, so $F'$ is a fragment of $G$ properly contained in $F$, a contradiction.

This completes the proof of Lemma \ref{subend}.
\end{proof}
\begin{lemma}\label{lem:add-vertex-outside-end}
Let $G$ be a graph with $\kappa(G)=k$, let $S$ be a $k$-cutset of $G$ and let $F$ be an end of $G$ to $S$. Let $Q$ be a set of new vertices such that every vertex of $Q$ has at least $k$ neighbors in $G$ and no neighbor in $F$, where the edges in $Q$ are arbitrary. Denote the resulting graph by $G'$. Then $\kappa(G')=k$, and $F$ is still an end of $G'$ to $S$.
\end{lemma}
\begin{proof}
Let $X\subseteq V(G')$ with $|X|\le k-1$, and let $X_G=X\cap V(G)$. Since $G$ is $k$-connected, $G-X_G$ is connected. For every $u\in Q\setminus X$, the vertex $u$ has at least $k$ neighbors in $V(G)$. Since $|X_G|\le k-1$, at least one of these neighbors lies in $V(G)\setminus X_G$. Thus every vertex of $Q\setminus X$ is adjacent to the connected graph $G-X_G$. Consequently, $G'-X$ is connected. Therefore, $G'$ is $k$-connected.

Since no vertex of $Q$ has a neighbor in $F$, we have $N_{G'}(F)=N_G(F)=S$. Hence $S$ remains a $k$-cut of $G'$. If $F$ contains a fragment $A$ of $G'$ properly, then $N_G(A)=N_{G'}(A)$, since no new vertex has a neighbor in $F$. Thus $A$ is also a fragment of $G$, contradicting the assumption that $F$ is an end of $G$. Therefore, $F$ remains an end of $G'$ to $S$.
\end{proof}
For a tree $T$ with bipartition $(X,Y)$, define $\max(T)=\max\{|X|,|Y|\}$.
\begin{lemma}\label{vT^*}
Let $G$ be a triangle-free graph and let $T\cong P_j^+(i)$ be a subgraph of $G$. Then $|N_G(v)\cap V(T)|\leq \max(T)$ for any $v\in V(G)$.
\end{lemma}
\begin{proof}
Choose an arbitrary vertex $v\in V(G)$. Since $G$ is triangle-free, $N_G(v)\cap V(T)$ is an independent set of $T$ (i.e., no two vertices in
$N_G(v)\cap V(T)$ are adjacent in $T$). Let $t$ be a leaf of $T$ such that $T-t$ is a path, and let $N_T(t)=\{t'\}$. Let $T-\{t,t'\}=P_1\cup P_2$, where $P_1$ and $P_2$ are vertex-disjoint paths, either of which may be empty. An empty path is understood to have order zero. Note that $\max(T)=\left \lceil \frac{|P_1|}{2} \right \rceil +\left \lceil \frac{|P_2|}{2} \right \rceil +1$.

If $t\notin N_G(v)$, since $T-\{t\}$ is a path, $|N_G(v)\cap V(T)|\leq\left \lceil \frac{|T|-1}{2}\right \rceil=\left \lceil \frac{|P_1|+|P_2|+1}{2}\right \rceil\leq \max(T)$. If $t\in N_G(v)$, then $t'\notin N_G(v)$ and thus $|N_G(v)\cap V(T)|\leq\left \lceil \frac{|P_1|}{2}\right \rceil+\left \lceil \frac{|P_2|}{2}\right \rceil+1=\max(T)$.

This completes the proof.
\end{proof}
\section{Proof of Theorem \ref{1tm-t-2}}\label{genthm}

\subsection{The case of $K_{1,3}$}\label{K13}
We first prove Theorem \ref{1tm-t-2} for $K_{1,3}$. Suppose, to the contrary, that there exists a $k$-connected graph $G$ satisfying $\delta(G)\ge \left\lfloor \frac{3k}{2}\right\rfloor+3$ and that
$U$ contains no subgraph $T\cong K_{1,3}$ such that $\kappa(G-V(T))\geq k$, where $U$ is an end of $G$ if $\kappa(G)=k$, and $U=G$ otherwise.
\begin{definition}
Let $u_1\in V(U)$, let $P_{u_1}=u_2u_3u_4$ be a path of order $3$ in $U-\{u_1\}$, let $H=G-(\{u_1\}\cup V(P_{u_1}))$. Then $(u_1,P_{u_1},F)$ is called a triple of $G$ if
\begin{enumerate}[label=(\roman*)]
\item $\kappa(H)=k$;
\item $F\subseteq U$ is an end of $H$;
\item $N_{G}(u_2)\cap V(F)=\emptyset$ and $N_{G}(u_4)\cap V(F)\neq\emptyset$.
\end{enumerate}
\end{definition}
We first establish the existence of a triple in $G$.
\begin{lemma}\label{Gtriple}
$G$ contains a triple.
\end{lemma}
\begin{proof}
To prove the lemma, we first prove the following claim.
\begin{claim}\label{P}
There exists a path $P\subseteq U$ of order $3$ satisfying $\kappa(G-V(P))=k$ and there exists an end $F\subseteq U$ of $G-V(P)$ such that $N_G(P)\cap V(F)\neq \emptyset$.
\end{claim}
\begin{proof}
We first consider the case $\kappa(G)\geq k+1$. Observe that $\delta(G)\geq \left\lfloor \frac{3k}{2}\right\rfloor+3\ge \left\lfloor \frac{3(k+1)}{2}\right\rfloor+1$. By Theorem \ref{mader1}, $G$ contains an edge $u_1u_2$ such that $\kappa(G-\{u_1,u_2\})\ge k+1$.

We shall show that for every $u_3\in N_G(u_2)\setminus\{u_1\}$, we have $\kappa(G-\{u_1,u_2,u_3\})=k$. Indeed, deleting one vertex from the $(k+1)$-connected graph $G-\{u_1,u_2\}$ gives $\kappa(G-\{u_1,u_2,u_3\})\ge k$. Suppose that there exists $u_3\in N_G(u_2)\setminus \{u_1\}$ such that $\kappa(G-\{u_1,u_2,u_3\})\ge k+1$. Since $\delta(G)>3$, we may choose $u_2'\in N_G(u_2)\setminus\{u_1,u_3\}$. Then $G[\{u_1,u_2,u_3,u_2'\}]$ contains a copy of $K_{1,3}$ and $\kappa(G-\{u_1,u_2,u_3,u_2'\})\ge k$, a contradiction.

Fix $u_3\in N_G(u_2)\setminus\{u_1\}$, and let $H_0=G-\{u_1,u_2,u_3\}$. Then $\kappa(H_0)=k$. Choose an end $F$ of $H_0$, and let $S=N_{H_0}(F)$. Since $\kappa(G)\ge k+1$, we must have $N_G(V(u_1u_2u_3))\cap V(F)\neq\emptyset$. Otherwise, $S$ is a $k$-cut of $G$, contradicting $\kappa(G)\geq k+1$. Thus, $u_1u_2u_3$ is the desired path.

It remains to consider the case $\kappa(G)=k$. Since $\delta(G)\geq \left\lfloor \frac{3k}{2}\right\rfloor+3$, there exists a path $P\subseteq U$ of order $3$ such that $\kappa(G-V(P))\geq k$. Since $|U|\geq \delta(G[V(U)])+1\geq (\left\lfloor \frac{3k}{2}\right\rfloor+3-k)+1\geq 4$, $V(U)\setminus V(P)\neq \emptyset$. By Lemma \ref{subend}, we obtain that $\kappa(G-V(P))=k$ and there exists an end $F\subseteq U$ of $G-V(P)$ satisfying $N_G(P)\cap V(F)\neq \emptyset$. Thus, $P$ is the desired path.
\end{proof}

Choose $P=u_1u_2u_3\subseteq U$ and an end $F\subseteq U$ of $G-V(P)$ satisfying $\kappa(G-V(P))=k$ and $N_G(P)\cap V(F)\neq \emptyset$. Let $H_1=G-V(P)$. Now $H_1$ is a graph satisfying $\kappa(H_1)=k$ and $\delta(H_1)\geq \delta(G)-|V(P)|\geq \left\lfloor \frac{3k}{2}\right\rfloor$. Hence, the graph obtained from $H_1$ by deleting an arbitrary vertex from $V(F)$ is still $k$-connected by Lemma \ref{mpath}.

We shall show that $N_G(u_2)\cap V(F)=\emptyset$. Suppose instead that $z\in N_G(u_2)\cap V(F)$. Then the subgraph of $G$ induced by $\{u_1,u_2,u_3,z\}$ contains a copy of $K_{1,3}$ with center $u_2$, while $G-\{u_1,u_2,u_3,z\}=H_1-\{z\}$ is $k$-connected. This contradicts our assumption. Therefore, $u_2\notin N_G(F)$, which implies $|F|\geq \delta(G[V(F)])+1\geq (\left\lfloor \frac{3k}{2}\right\rfloor+3-k-2)+1\geq 2$.

Consequently, $N_G(\{u_1,u_3\})\cap V(F)\neq\emptyset$. Without loss of generality, we may assume that $N_G(u_3)\cap V(F)\neq\emptyset$. Choose $u_4\in N_G(u_3)\cap V(F)$. Let $H=H_1-\{u_4\}$. Note that $V(F)\setminus \{u_4\}\neq \emptyset$ since $|F|\geq 2$. Combined with $\kappa(H)\geq k$, we obtain that $\kappa(H)=k$ and there exists an end $F'\subseteq F$ of $H$ such that $N_G(u_4)\cap V(F')\neq \emptyset$ by Lemma \ref{subend}. Since $F'\subseteq F$, we have $N_G(u_2)\cap V(F')=\emptyset$. Therefore $(u_1,u_2u_3u_4,F')$ is a triple of $G$.
\end{proof}

We now choose a triple $(u_1,u_2u_3u_4,F)$ for which $|F|$ is minimum among all triples in $G$. Let $H=G-\{u_1,u_2,u_3,u_4\}$, and let $S=N_H(F)$. Then $\kappa(H)=k$ and $F$ is an end of $H$ satisfying $N_G(u_2)\cap V(F)=\emptyset$ and $N_G(u_4)\cap V(F)\neq \emptyset$ by the definition of triple.

Note that every vertex of $\{u_1,u_2,u_3,u_4\}$ has at most two neighbors in the other three vertices; otherwise, $G[\{u_1,u_2,u_3,u_4\}]$ contains a copy of $K_{1,3}$ whose deletion preserves the required connectivity, a contradiction. We next analyze how the vertices $u_1,u_2,u_3,u_4$ are adjacent to $H$ and determine the resulting structure of $H$.
\begin{lemma}\label{u3F}
$N_G(u_3)\cap V(F)\neq \emptyset$.
\end{lemma}
\begin{proof}
Suppose, to the contrary, that $N_G(u_3)\cap V(F)=\emptyset$. Then $N_G(F)\cap\{u_2,u_3\}=\emptyset$ and thus $|F|\geq \delta (G[V(F)])+1\geq (\left\lfloor \frac{3k}{2}\right\rfloor+3-2-k)+1\geq 2$. Let $H'=G[V(H)\cup \{u_2\}]$. Since $|N_G(u_2)\cap V(H)|\geq \left\lfloor \frac{3k}{2}\right\rfloor+3-3\geq k$ and $(N_G(u_2)\cap V(H))\subseteq V(H)\setminus V(F)$, $S$ is still a $k$-cutset of $H'$ and $F$ is an end of $H'$ to $S$ by Lemma \ref{lem:add-vertex-outside-end}.

Recall that $N_G(u_4)\cap V(F)\neq \emptyset$. Take $u_5\in N_G(u_4)\cap V(F)$. Since $\delta(H')\geq \delta(G)-3\geq \left\lfloor \frac{3k}{2}\right\rfloor$, $\kappa(H'-\{u_5\})\geq k$ by Lemma \ref{mpath}. Moreover, since $|V(F)\setminus \{u_5\}|\geq 2-1>0$, $\kappa(H'-\{u_5\})=k$ and there exists an end $F'\subseteq F$ with $N_G(u_5)\cap V(F')\neq \emptyset$ by Lemma \ref{subend}. Since $N_G(u_3)\cap V(F')=\emptyset$ since $F'\subseteq F$ and $N_G(u_3)\cap V(F)=\emptyset$, $(u_1,u_3u_4u_5,F')$ is a triple in $G$ satisfying $|F'|<|F|$, contradicting the choice of $(u_1,u_2u_3u_4,F)$.

The proof of Lemma \ref{u3F} is complete.
\end{proof}
For every $u\in N_G(u_3)\cap V(F)$, the graph $U[\{u_2,u_3,u_4,u\}]$ contains a copy of $K_{1,3}$. Thus $H-\{u\}$ is not $k$-connected; otherwise, adjoining $u_1$, which has at least $\delta(G)-3\geq k$ neighbors in $H-\{u\}$, gives a contradiction. Since $H$ is $k$-connected while $H-\{u\}$ is not, there exists a cut set $R$ of $H-\{u\}$ with $|R|\leq k-1$. Then $R\cup \{u\}$ is a $k$-cutset of $H$. Since $u\in V(F)\subseteq V(H)\setminus S$, we have $R\cup \{u\}\neq S$.
\begin{lemma}\label{lem:structure-of-k-cut}
Let $u\in N_G(u_3)\cap V(F)$, and let $S'$ be a $k$-cutset of $H$ containing $u$. Then $H-S'$ has exactly two components. Moreover, one component is a clique of order $\lfloor \frac{k}{2}\rfloor$, every vertex of which is adjacent to each of $u_1,u_2,u_3,u_4$ in $G$.
\end{lemma}
\begin{proof}
Let $F'$ be a fragment of $H$ to $S'$. Let $\bar{F'}=H-(S'\cup V(F'))$, and let $\bar{F}=H-(S\cup V(F))$. We use the following notation for the nine intersections:
\begin{align*}
A_1=F\cap F', A_2=S\cap F',A_3=\bar F\cap F',\\
A_4=F\cap S',A_5=S\cap S',A_6=\bar F\cap S',\\
A_7=F\cap \bar{F'},A_8=S\cap \bar{F'},A_9=\bar F\cap \bar{F'}.
\end{align*}
\begin{claim}\label{k+1}
If $A_1\neq\emptyset$, then $|A_2\cup A_4\cup A_5|\geq k+1$. Similarly, if $A_7\neq\emptyset$, then $|A_4\cup A_5\cup A_8|\geq k+1$.
\end{claim}
\begin{proof}
Suppose that $A_1\neq\emptyset$. Since $V(\bar F)\neq\emptyset$, $A_2\cup A_4\cup A_5$ is a vertex cut of $H$. Hence, as $\kappa(H)=k$, $|A_2\cup A_4\cup A_5|\geq k$. If equality holds, then $A_1$ is a fragment of $H$ to $A_2\cup A_4\cup A_5$, and thus contains an end $F_1$ of $H$. Since $u\in A_4$, $A_1\subsetneq F$ and thus $F_1\subseteq A_1\subsetneq F$, this contradicts the choice of $F$ as an end. Therefore, $|A_2\cup A_4\cup A_5|\geq k+1$.

The second assertion follows similarly.
\end{proof}
\begin{claim}\label{k+1empty}
If $|A_2\cup A_4\cup A_5|\ge k+1$, then $A_9=\emptyset$. Similarly, if $|A_4\cup A_5\cup A_8|\ge k+1$, then $A_3=\emptyset$.
\end{claim}
\begin{proof}
Suppose, to the contrary, that $A_9\neq\emptyset$. Then $A_5\cup A_6\cup A_8$ is a vertex cut of $H$, and hence
$|A_5\cup A_6\cup A_8|\geq k$. Together with the above inequality, we obtain $|S|+|S'|=(|A_2\cup A_4\cup A_5|)+(|A_5\cup A_6\cup A_8|)>2k$, a contradiction. Therefore, $A_9=\emptyset$.

The other statement follows symmetrically.
\end{proof}
\begin{claim}\label{A3A9empty}
If $A_3=A_9=\emptyset$, then one of the following holds:
\begin{itemize}
    \item[(i)] $H[A_6]\cong K_{\left\lfloor \frac{k}{2}\right\rfloor}$ and $\{u_1,u_2,u_3,u_4\}\subseteq N_G(u)$ for every $u\in V(A_6)$,
    \item[(ii)] $|A_6|>|A_4|$.
\end{itemize}
\end{claim}
\begin{proof}
Since $A_3=A_9=\emptyset$, we have $\bar {F}=A_6\neq\emptyset$ and $N_G(u)\subseteq \{u_1,u_2,u_3,u_4\}\cup S\cup V(A_6)$
for every $u\in V(A_6)$. Hence $|A_6|\geq \delta(H[A_6])+1\geq \delta(G)-4-k+1\geq \left\lfloor\frac{3k}{2}\right\rfloor+3-4-k+1\geq \left\lfloor\frac{k}{2}\right\rfloor$. If $|A_6|\geq \left\lfloor\frac{k}{2}\right\rfloor+1$, since $|A_4\cup A_5\cup A_6|=k$, we obtain $|A_4|\leq k-|A_6|\leq k-(\left\lfloor\frac{k}{2}\right\rfloor+1)\leq \left\lfloor\frac{k}{2}\right\rfloor<|A_6|$, and thus $(ii)$ holds.

It remains to consider $|A_6|=\left\lfloor\frac{k}{2}\right\rfloor$. For any $u\in V(A_6)$,
\[
|N_G(u)\cap\{u_1,u_2,u_3,u_4\}|+d_{H[A_6]}(u)
\geq d_G(u)-|S|
\geq \left\lfloor \frac{k}{2}\right\rfloor+3.
\]
Since the left-hand side is at most $4+\left(|A_6|-1\right)=\left\lfloor \frac{k}{2}\right\rfloor+3$, equality holds throughout. Hence every vertex of $A_6$ is adjacent to each of $u_1,u_2,u_3,u_4$, and $d_{H[A_6]}(u)=\left\lfloor\frac{k}{2}\right\rfloor-1$. Therefore, $H[A_6]\cong K_{\left\lfloor\frac{k}{2}\right\rfloor}$ and $(i)$ holds.
\end{proof}
The proof of the following claim is analogous to that of Claim \ref{A3A9empty} and is therefore omitted.
\begin{claim}\label{A7A9empty}
If $A_7=A_9=\emptyset$, then one of the following holds:
\begin{itemize}
    \item[(i)] $H[A_8]\cong K_{\left\lfloor \frac{k}{2}\right\rfloor}$ and $\{u_1,u_2,u_3,u_4\}\subseteq N_G(u)$ for every $u\in V(A_8)$,
    \item[(ii)] $|A_8|>|A_2|$.
\end{itemize}
\end{claim}
\begin{claim}\label{A3A9nbe}
If $A_3=A_9=\emptyset$, then $|A_6|>|A_4|$.
\end{claim}
\begin{proof}
Suppose that $A_3=A_9=\emptyset$. By Claim \ref{A3A9empty}, if $(i)$ holds, then $\bar {F}=A_6$ and $H[A_6]\cong K_{\left\lfloor \frac{k}{2}\right\rfloor}$.
Since $N_G(u_2)\cap V(F)=\emptyset$ and $u_2$ has at most two neighbors in $\{u_1,u_3,u_4\}$, we have $d_G(u_2)\leq 2+|S|+|A_6|=\left\lfloor\frac{3k}{2}\right\rfloor+2<\delta(G)$, a contradiction. Hence $|A_6|>|A_4|$ holds.
\end{proof}

We divide the remainder of the proof into the following three cases.

{\bf Case 1.} $A_1\neq\emptyset$ and $A_7\neq\emptyset$.

Since $A_1\neq\emptyset$, Claim \ref{k+1} gives $|A_2\cup A_4\cup A_5|\geq k+1$, and hence $A_9=\emptyset$ by Claim \ref{k+1empty}. Similarly, $A_7\neq\emptyset$ implies $|A_4\cup A_5\cup A_8|\geq k+1$ and thus $A_3=\emptyset$.

Therefore, $A_3=A_9=\emptyset$, and Claim \ref{A3A9nbe} gives $|A_6|>|A_4|$. It follows that $|A_5\cup A_6\cup A_8|>|A_4\cup A_5\cup A_8|\geq k+1$. Consequently, $|S|+|S'|=|A_2\cup A_4\cup A_5|+|A_5\cup A_6\cup A_8|>2k$, a contradiction.

{\bf Case 2.} $A_1=A_7=\emptyset$.

Since $N_G(u_2)\cap V(F)=\emptyset$, $|N_G(A_4)\cap \{u_1,u_2,u_3,u_4\}|\leq 3$. Combined with $A_1=A_7=\emptyset$, $|A_4|\geq \delta(H[A_4])+1\geq \delta(G)-3-k+1\geq \left\lfloor \frac{k}{2}\right\rfloor+1$. Then $|A_6|\leq |S'|-|A_4|\leq k-(\left\lfloor \frac{k}{2}\right\rfloor+1)\leq \left\lfloor \frac{k}{2}\right\rfloor<|A_4|$, and thus $A_3$ and $A_9$ are not both empty by Claim \ref{A3A9nbe}.

By symmetry of $A_3$ and $A_9$, we may assume $A_3\neq \emptyset$, which implies that $A_2\cup A_5\cup A_6$ is a cut set of $H$. Hence, $|A_2\cup A_5\cup A_6|\geq k$. Since $|A_4|>|A_6|$, we have $|A_2\cup A_4\cup A_5|>|A_2\cup A_5\cup A_6|\geq k$, and thus $A_9=\emptyset$ by Claim \ref{k+1empty}.

Now $A_7=A_9=\emptyset$. Suppose that $|A_8|>|A_2|$. Then $|A_5\cup A_6\cup A_8|>|A_2\cup A_5\cup A_6|\geq k$. Recall that $|A_2\cup A_4\cup A_5|>k$. Thus, $|S|+|S'|=|A_2\cup A_4\cup A_5|+|A_5\cup A_6\cup A_8|>2k$, a contradiction. Therefore, $\bar{F'}=H[A_8]\cong K_{\left\lfloor\frac{k}{2}\right\rfloor}$ and every vertex of $A_8$ is adjacent in $G$ to each of $u_1,u_2,u_3,u_4$ by Claim \ref{A7A9empty}.

Since $|A_4|\geq \left\lfloor \frac{k}{2} \right\rfloor+1$ and $|A_8|=\left\lfloor \frac{k}{2} \right\rfloor$, we have $|A_4\cup A_5\cup A_8|\geq k$. Since $2k=|S|+|S'|=|A_4\cup A_5\cup A_8|+|A_2\cup A_5\cup A_6|$, it follows that $|A_2\cup A_5\cup A_6|\leq k$. Together with $|A_2\cup A_5\cup A_6|\geq k$, this yields $|A_2\cup A_5\cup A_6|=k$. Moreover, since $|A_4\cup A_5\cup A_6|=k$, we have $|A_2|=|A_4|>0$. Thus $A_2\cup A_5\cup A_6$ is a minimum vertex cut of $H$, and hence every component of $H[A_3]$ is adjacent to every vertex of $A_2$. Consequently, $F'=H[A_1\cup A_2\cup A_3]=H[A_2\cup A_3]$ is connected. Hence, $H-S'$ has exactly the two components $F'$ and $\bar {F'}$, and the lemma follows.

{\bf Case 3.} Exactly one of $A_1$ and $A_7$ is nonempty.

By the symmetry of $A_1$ and $A_7$, we may assume, without loss of generality, that $A_1\neq\emptyset$ and $A_7=\emptyset$. By Claims \ref{k+1} and \ref{k+1empty}, we obtain $|A_2\cup A_4\cup A_5|\geq k+1$ and $A_9=\emptyset$. Now $A_7=A_9=\emptyset$.

We first show that $|A_4\cup A_5\cup A_8|=k$. Since $A_7=A_9=\emptyset$, by Claim \ref{A7A9empty}, one of the following alternatives holds: $|A_8|>|A_2|$, or $|A_8|=\left\lfloor\frac{k}{2}\right\rfloor$ and $|A_2|=|S|-|A_5|-|A_8|\leq \left\lfloor\frac{k}{2}\right\rfloor+1$. In either case, we have $|A_8|\geq |A_2|-1$. Since $|A_2\cup A_4\cup A_5|\geq k+1$, it follows that $|A_4\cup A_5\cup A_8|\geq |A_2\cup A_4\cup A_5|-1\geq k$. Suppose that $|A_4\cup A_5\cup A_8|>k$. Then $A_3=\emptyset$ by Claim \ref{k+1empty}. Hence, $A_3=A_9=\emptyset$ and thus $|A_6|>|A_4|$ by Claim \ref{A3A9nbe}. Then we have $|A_2\cup A_5\cup A_6|> |A_2\cup A_5\cup A_4|\geq k+1$. Together with $|A_4\cup A_5\cup A_8|>k$, this yields $|S|+|S'|=|A_4\cup A_5\cup A_8|+|A_2\cup A_5\cup A_6|>2k$, a contradiction. Therefore, $|A_4\cup A_5\cup A_8|=k$.

Since $|A_2\cup A_4\cup A_5|\geq k+1$ and $|A_4\cup A_5\cup A_8|=k$, we have $|A_2|>|A_8|$. Hence, by Claim \ref{A7A9empty}, $H[A_8]\cong K_{\left\lfloor\frac{k}{2}\right\rfloor}$ and every vertex of $A_8$ is adjacent to each of $u_1,u_2,u_3,u_4$ in $G$. Moreover, since $|A_2\cup A_5\cup A_8|=k$, $|A_2|>|A_8|$ and $|A_8|=\left\lfloor\frac{k}{2}\right\rfloor$, we obtain $|A_2|=\left\lfloor\frac{k}{2}\right\rfloor+1$ and $|A_5|=0$. In particular, $k$ is odd.

Since $|A_4\cup A_5\cup A_8|=k$, together with $|A_8|=\left\lfloor\frac{k}{2}\right\rfloor$ and $|A_5|=0$, we obtain $|A_4|=\left\lfloor\frac{k}{2}\right\rfloor+1$. Consequently, $|A_2\cup A_4\cup A_5|=k+1$. Let $C$ be a component of $H[A_1]$. If some vertex of $A_2$ has no neighbor in $C$, then $|N_H(C)|\le |A_2\cup A_4\cup A_5|-1=k$. Since $H$ is $k$-connected and $C\subsetneq F$, $C$ is a fragment of $H$ properly contained in the end $F$, a contradiction.
Thus every component of $H[A_1]$ has a neighbor at each
vertex of $A_2$. On the other hand, since $|A_4\cup A_5\cup A_8|+|A_2\cup A_5\cup A_6|=|S|+|S'|=2k$, we obtain $|A_2\cup A_5\cup A_6|=k$. If $A_3\neq\emptyset$, then $A_2\cup A_5\cup A_6$ is a $k$-cut of $H$, and hence every component of $H[A_3]$ has a neighbor in $A_2$. Therefore, $F'=H[A_1\cup A_2\cup A_3]$ is connected. Since
$\bar{F'}=H[A_8]\cong K_{\left\lfloor \frac{k}{2}\right\rfloor}$, the graph $H-S'$ has exactly the two components $F'$ and $\bar{F'}$, and the lemma follows.
\end{proof}

We are now ready to establish the $K_{1,3}$ case of Theorem \ref{1tm-t-2}. Choose a vertex $u\in N_G(u_3)\cap V(F)$, which exists by Lemma \ref{u3F}. Since $u_2u_3,u_3u_4,u_3u\in E(G)$, the graph $G[\{u_2,u_3,u_4,u\}]$ contains a copy of $K_{1,3}$ with center $u_3$. By the assumption that $U$ contains no subgraph $T\cong K_{1,3}$ such that $\kappa(G-V(T))\geq k$, we have $\kappa(G-\{u_2,u_3,u_4,u\})<k$.
Let $H'=G[(\{u_1\}\cup V(H))-\{u\}]=G-\{u_2,u_3,u_4,u\}$.
Then $\kappa(H')<k$. Since $\kappa(H)=k$, the decrease of the connectivity is caused by deleting the vertex $u$. Hence there exists a $k$-cutset $S'$ of $H$ containing $u$ such that some component of $H-S'$ has no neighbor to $u_1$.

By Lemma \ref{lem:structure-of-k-cut}, $H-S'$ has exactly two components $C_1$ and $C_2$, where $C_1\cong K_{\left\lfloor \frac{k}{2} \right\rfloor}$ and every vertex of $C_1$ is adjacent to each vertex of $u_1,u_2,u_3,u_4$ in $G$. Moreover, $C_2$ contains no neighbor of $u_1$. Therefore, $|N_G(u_1)\cap V(H)|\leq |S'|+|C_1|=k+\left\lfloor \frac{k}{2}\right\rfloor=\left\lfloor\frac{3k}{2}\right\rfloor$.

Since $u_1$ has at most two neighbors in $\{u_2,u_3,u_4\}$, we have $d_G(u_1)\leq |N_G(u_1)\cap V(H)|+2\leq \left\lfloor\frac{3k}{2}\right\rfloor+2 <\left\lfloor\frac{3k}{2}\right\rfloor+3$, contradicting the assumption that $\delta(G)\geq \left\lfloor\frac{3k}{2}\right\rfloor+3$.

This proves Theorem \ref{1tm-t-2} for $K_{1,3}$.
\hfill$\qed$

\subsection{Completion of the proof}\label{genfull}
\begin{proof}
Suppose, to the contrary, that Theorem \ref{1tm-t-2} is false, and let $P_s^+(i)$ be a counterexample of minimum order. If $i=1$ or $i=s-1$, then $P_s^+(i)$ is a path. The required conclusion follows from Theorem \ref{mader1} when $\kappa(G)>k$, and from Lemma \ref{mpath} when $\kappa(G)=k$. Hence $2\leq i\leq s-2$.

Let $G$ be a corresponding $k$-connected graph with $\delta(G)\geq \left\lfloor\frac{3k}{2}\right\rfloor+s-1$. If $\kappa(G)=k$, let $F^*$ be an end of $G$ containing no subgraph $H\cong P_s^+(i)$ with $\kappa(G-V(H))\geq k$, and let $U=F^*$. Otherwise, let $U=G$. By the minimality of $s$,
together with Theorem \ref{mader1} and Lemma \ref{mpath}, the conclusion holds for every pendant extension of a path of order less than $s$.

Let $c$ be the vertex of degree three in $P_s^+(i)$, and let $L_j=v_j^1\cdots v_j^{l_j}$ for $j\in[3]$ be the three components of $P_s^+(i)-c$, where $v_j^1$ is adjacent to $c$ and $v_j^0=c$. Without loss of generality, assume that $l_1\geq l_2\geq l_3=1$. Thus $l_1+l_2+l_3=s-1$. By Subsection \ref{K13}, we may assume that $s\geq5$, and hence $l_1\geq2$.

By the minimality of $s$, $U$ contains a subgraph $T\cong P_s^+(i)-v_1^{l_1}$ such that $\kappa(G-V(T))\geq k$. Let $\phi$ be the corresponding
isomorphism, and let $\phi(c)=c'$ and $\phi(v_j^r)=u_j^r$. Write $L_1'=u_1^1\cdots u_1^{l_1-1}$, $L_2'=u_2^1\cdots u_2^{l_2}$, and $L_3'=u_3^1$. Let $u_1^0=c'$ and $G_1=G-V(T)$.

We have $\kappa(G_1)=k$. Indeed, if $\kappa(G)>k$ and $\kappa(G_1)>k$, then a neighbor of $u_1^{l_1-1}$ outside $V(T)$ extends $T$ to the required copy of $P_s^+(i)$. If $\kappa(G)=k$, then $|F^*|\geq\delta(G)-k+1>|T|$, so $N_G(F^*)$ remains a $k$-cutset of $G_1$ and thus $\kappa(G_1)=k$.

Choose an end $F$ of $G_1$ such that $V(F)\subseteq V(U)$ and $N_G(T)\cap V(F)\neq\emptyset$. When $\kappa(G)=k$, such an end exists by Lemma \ref{subend}. When $\kappa(G)>k$, any end $F'$ of $G_1$ has the required property, since otherwise $N_G(F')$ would be a $k$-cutset of $G$.

Since every vertex of $T$ has at least $\delta(G)-(|T|-1)\geq\left\lfloor\frac{3k}{2}\right\rfloor+1\geq k$ neighbors in $G_1$, it follows from Lemma \ref{lem:add-vertex-outside-end} that, for any $X\subseteq V(T)$ with $N_G(X)\cap V(F)=\emptyset$, we have $\kappa(G[V(G_1)\cup X])=k$, and $F$ remains an end of $G[V(G_1)\cup X]$.

To derive a contradiction, we distinguish three cases based on the adjacency relations between $T$ and $F$.
Before proceeding to the case analysis, we first establish the following claim, which will be a useful tool in our subsequent arguments.
\begin{claim}\label{gl_1}
For any $j\in[l_1]$, if $U$ contains a subgraph $T'\cong P_s^+(i)-V(L_1[v_1^j,v_1^{l_1}])$ with an isomorphism $\tau: P_s^+(i)-V(L_1[v_1^j,v_1^{l_1}])\to T'$ such that $\kappa(G-V(T'))=k$ and $N_G(\tau(v_1^{j-1}))\cap V(F_j)\neq \emptyset$, where $F_j$ is an end of $G-V(T')$ with $V(F_j)\subseteq V(U)$, then $U$ contains a subgraph $H\cong P_s^+(i)$ such that $\kappa(G-V(H))\geq k$.
\end{claim}
\begin{proof}
Choose $x\in N_G(\tau(v_1^{j-1}))\cap V(F_j)$. Since $|T'|=s-(l_1-j+1)$, we have $\delta(G-V(T'))\geq\delta(G)-|T'|\geq \left\lfloor\frac{3k}{2}\right\rfloor+(l_1-j+1)-1$. Since $\kappa(G-V(T'))=k$ and $F_j$ is an end of $G-V(T')$, there exists a path $P_x$ of order $l_1-j+1$ starting from $x$ in $F_j$ such that $\kappa(G-(V(T')\cup V(P_x)))\geq k$ by Lemma \ref{mpath}. Joining $P_x$ to $T'$ by the edge $\tau(v_1^{j-1})x$ gives the required subgraph $H$.
\end{proof}
We now proceed with the case analysis.

{\bf Case 1.} $N_G(c'L_1')\cap V(F)\neq\emptyset$.

Let $j$ be the largest index such that $N_G(u_j^1)\cap V(F)\neq\emptyset$, where $0\leq j\leq l_1-1$. Let $X=\{u^r_1:j+1\le r\le l_1-1\}$ and $T'=T-X$. By the choice of $j$, $N_G(X)\cap V(F)=\emptyset$. Thus $\kappa(G-V(T'))=k$ and $F$ remains an end of $G-V(T')$. Applying Claim \ref{gl_1} with $j+1$ gives a contradiction.

{\bf Case 2.} $N_G(u_3^1)\cap V(F)\neq\emptyset$.

By Case 1, $N_G(c'L_1')\cap V(F)=\emptyset$. Let $T'=T-\{u_1^r:2\leq r\leq l_1-1\}$. Then $\kappa(G-V(T'))=k$ and $F$ remains an end of $G-V(T')$. Interchanging the roles of $u_1^1$ and $u_3^1$ in $T'$, Claim \ref{gl_1} with $j=2$ gives a contradiction.

{\bf Case 3.} $N_G(L_2')\cap V(F)\neq\emptyset$.

Let $j$ be the largest index such that $N_G(u_2^j)\cap V(F)\neq\emptyset$. Then $L_2'[u_2^j,u_2^1]c'L_1'$ is a path of order $j+l_1\geq l_1+1\geq l_2+1$.

Suppose first that $j+l_1=l_2+1$. Then $l_1=l_2$ and $j=1$. Recall that $l_1\geq 2$. Then $l_2\geq 2$ and $u_2^2$ exists. Let $R=u_2^2u_2^1c'L_1'$.
Then $|R|=l_2+2$, and no vertex of $T-V(R)$ has a neighbor in $F$ by Cases 1, 2 and the choice of $j$. Thus $\kappa(G-V(R))=k$ and $F$ remains an end of $G-V(R)$. Viewing $R$ as a copy of $P_s^+(i)-V(L_1)$ with $c$ mapped to $u_2^1$, applying Claim \ref{gl_1} with $j=1$ gives a contradiction.

Hence $j+l_1\geq l_2+2$. Let $Q$ be the subpath of $L_2'[u_2^j,u_2^1]c'L_1'$ of order $l_2+2$ starting at $u_2^j$. Since $j\leq l_2$, the other endvertex of $Q$ lies in $L_1'$. By Cases 1 and 2 and the choice of $j$, $N_G(T-V(Q))\cap V(F)=\emptyset$. By Lemma \ref{lem:add-vertex-outside-end}, $\kappa(G-V(Q))=k$ and $F$ remains an end of $G-V(Q)$. Moreover, exactly one endvertex of $Q$ has a neighbor in $F$.

Choose a path $P=x_1\cdots x_{l_2+2}$ in $U$ and an end $F_P$ of $G-V(P)$ such that $\kappa(G-V(P))=k$, $V(F_P)\subseteq V(U)$, and exactly one endvertex of $P$ has a neighbor in $F_P$, with $|F_P|$ as small as possible. Without loss of generality, we assume $N_G(x_1)\cap V(F_P)=\emptyset$ and $N_G(x_{l_2+2})\cap V(F_P)\neq\emptyset$.

By Claim \ref{gl_1} with $j=1$, neither $x_2$ nor $x_{l_2+1}$ has a neighbor in $F_P$; otherwise, viewing $P$ as a copy of $P_s^+(i)-V(L_1)$ with either vertex corresponding to $c$ gives a contradiction.

Choose $y\in N_G(x_{l_2+2})\cap V(F_P)$. Since $\delta(G-V(P))\geq\left\lfloor\frac{3k}{2}\right\rfloor+l_1-1 \geq\left\lfloor\frac{3k}{2}\right\rfloor+1$, Lemma \ref{mpath} gives $\kappa(G-(V(P)\cup \{y\}))\geq k$. Since $|F_P|\geq\delta(G-V(P))-k+1\geq 2$, $\kappa(G-(V(P)\cup \{y\}))=k$ and $G-(V(P)\cup \{y\})$ has an end $F'\subseteq F_P\subseteq U$ such that $N_G(y)\cap V(F')\neq\emptyset$ by Lemma \ref{subend}.

Let $P'=x_2\cdots x_{l_2+2}y$. The vertex $x_1$ has no neighbor in $F'$ and has at least $\delta(G)-(l_2+2)\geq\left\lfloor\frac{3k}{2}\right\rfloor+l_1-1\geq k$ neighbors in $G-(V(P)\cup \{y\})$. Thus, by Lemma \ref{lem:add-vertex-outside-end}, adjoining $x_1$ to $G-(V(P)\cup \{y\})$ gives $\kappa(G-V(P'))=k$ and $F'$ remains an end of $G-V(P')$. Since $N_G(x_2)\cap V(F_P)=\emptyset$ and $F'\subseteq F_P$, exactly one endvertex of $P'$ has a neighbor in $F'$. This contradicts the choice of $(P,F_P)$, as $|F'|<|F_P|$.
\end{proof}
\section{Proof of Theorem \ref{triangle}}\label{trithm}
\subsection{The case of $K_{1,3}$}\label{tri4}
%For any non-path $T^*=(U_1,U_2)$ with $|U_1|\le |U_2|$, let $t$ be a pendant vertex of $T^*$ and $t'=N_T(t)$. Then $T^*-t$ is a path and $T^*-\{t,t'\}$ consists of two paths. Let $P_1,P_2$ be the two paths of $T^*-\{t,t'\}$. Then by the parity of $|T^*|$, $|P_1|$ and $|P_2|$, we have the following facts:\\
%$(i)$ If $|T^*|$ is odd, then $|P_1|$ and $|P_2|$ have different parity, $|U_1|=\frac{|T^*|-1}{2}$ and $|U_2|=\frac{|T^*|+1}{2}$;\\
%$(ii)$ If $|T^*|$ is even and $|P_1|$ and $|P_2|$ are both even, then $|U_1|=|U_2|=\frac{|T^*|}{2}$;\\
%$(iii)$ If $|T^*|$ is even and $|P_1|$ and $|P_2|$ are both odd, then  $|U_1|=\frac{|T^*|}{2}-1$ and $|U_2|=\frac{|T^*|}{2}+1$.
\begin{proof}
We first prove Theorem \ref{triangle} for $K_{1,3}$. Let $U$ be an end of $G$ if $\kappa(G)=k$, and let $U=G$ otherwise. Suppose, for contradiction, that $U$ contains no subgraph $T\cong K_{1,3}$ such that $\kappa(G-V(T))\geq k$. If $\kappa(G)\geq k+1$, since $\delta(G)\geq k+4> (k+1)+\frac{3+1}{2}$, there exists a path $u_1u_2u_3$ in $G$ such that $\kappa(G-V(u_1u_2u_3))\geq k+1$ by Theorem \ref{F1}. As $|N_G(u_2)\setminus \{u_1,u_3\}|\geq \delta(G)-2>0$, there exists a neighbor $u_2'$ of $u_2$ in $G$ distinct from $u_1$ and $u_3$. Clearly, $G[\{u_1,u_2,u_3,u_2'\}]$ contains a copy of $K_{1,3}$, and $\kappa(G-\{u_1,u_2,u_3,u_2'\})\ge k$, a contradiction. Hence, $\kappa(G)=k$ and now $U$ is an end of $G$.

We first show the following observation. Let $u\in V(U)$ and let $P=p_1p_2p_3$ be a path in $U-\{u\}$ such that the connectivity of $H=G-(V(P)\cup\{u\})$ is $k$ and $H$ contains an end $F\subseteq U$. Then $N_G(p_2)\cap V(F)=\emptyset$. Otherwise, choose $z\in N_G(p_2)\cap V(F)$. Since $G$ is triangle-free, every vertex of $H$ has at most three neighbors in $V(P)\cup\{u\}$. Hence, $\delta(H)\geq\delta(G)-3\geq k+1$ and thus $\kappa(H-\{z\})\geq k$ by Theorem \ref{F1}. Since $u$ has at least $\delta(G)-4\geq k$ neighbors in $H-\{z\}$, adjoining $u$ implies $\kappa(G-(V(P)\cup\{z\}))\geq k$. However, $G[V(P)\cup\{z\}]$ contains a copy of $K_{1,3}$ in $U$, a contradiction. Thus, this observation holds.

Since $\delta(G)\geq k+4$, by Theorem \ref{F1}, there is a path $Q\subseteq U$ of order $4$ such that $\kappa(G-V(Q))\geq k$. Since $|U|\geq \delta(G)-k+1\geq 5$, $V(U-V(Q))\neq \emptyset$ and thus $\kappa(G-V(Q))=k$ and there exists an end $F\subseteq U$ of $G-V(Q)$ such that $N_G(Q)\cap V(F)\neq \emptyset$ by Lemma \ref{subend}. Let $Q=q_1q_2q_3q_4$. Applying the observation to $(q_4,q_1q_2q_3)$ and $(q_1,q_2q_3q_4)$, we obtain $N_G(\{q_2,q_3\})\cap V(F)=\emptyset$. Recall that $N_G(Q)\cap V(F)\neq \emptyset$. Without loss of generality, we assume $N_G(q_4)\cap V(F)\neq \emptyset$.

Call a pair $(u,P_u)$ admissible if $u\in V(U)$, $P_u$ is a path of order three in $U-\{u\}$, and $\kappa(G-(V(P_u)\cup\{u\}))=k$. Moreover, an admissible pair $(u,P_u)$ is said to be witnessed by an end $F_u$ if it satisfies the conditions below: $F_u$ is an end of $G-(V(P_u)\cup\{u\})$ with $V(F_u)\subseteq V(U)$ and $N_G(F_u)\cap V(P_u)$ consists of exactly one endvertex of $P_u$. The pair $(q_1,q_2q_3q_4)$ is admissible and is witnessed by the end $F$.

Among all such pairs, choose $(v_1,v_2v_3v_4)$  and $F_{v_1}$ with $|F_{v_1}|$ as small as possible. Without loss of generality, we assume $N_G(v_2)\cap V(F_{v_1})=\emptyset$ and $N_G(v_4)\cap V(F_{v_1})\neq \emptyset$. Let $H=G-\{v_1,v_2,v_3,v_4\}$ and choose $v_5\in N_G(v_4)\cap V(F_{v_1})$. Applying the observation to $(v_1,v_2v_3v_4)$, we obtain $N_G(v_3)\cap V(F_{v_1})=\emptyset$.  Since $G$ is triangle-free, $\delta(H)\geq\delta(G)-3\geq k+1$. By Theorem \ref{F1}, $\kappa(H-\{v_5\})\geq k$. Moreover, since $|F_{v_1}|\geq\delta(H)-k+1\geq2$, by Lemma \ref{subend}, $\kappa(H-\{v_5\})=k$ and $H-\{v_5\}$ has an end $F'\subseteq F_{v_1}$ such that $N_G(v_5)\cap V(F')\neq\emptyset$.

Since $v_2$ has at least $\delta(G)-4\geq k$ neighbors in $H-\{v_5\}$ and $N_G(v_2)\cap V(F')=\emptyset$, by Lemma \ref{lem:add-vertex-outside-end}, the connectivity of $H'=G-(\{v_1\}\cup V(v_3v_4v_5))=G[V(H-\{v_5\})\cup \{v_2\}]$ is $k$ and $F'$ is still an end of $H'$. Applying the observation to $(v_1,v_3v_4v_5)$, we obtain $N_G(v_4)\cap V(F')=\emptyset$. Since $N_G(v_3)\cap V(F_{v_1})=\emptyset$ and $F'\subseteq F_{v_1}$, it follows that $N_G(F')\cap\{v_3,v_4,v_5\}=\{v_5\}$. Thus $(v_1,v_3v_4v_5)$ and $F'$ satisfy the same conditions, contradicting the choice of $F_{v_1}$ as $|F'|<|F_{v_1}|$.
\end{proof}
\subsection{Completion of the proof}\label{trifull}
\begin{proof}
Suppose, to the contrary, that Theorem \ref{triangle} is false. Let $P_s^+(i)$ be a counterexample of minimum order. By Theorem \ref{F1} and Subsection \ref{tri4}, we may assume that $P_s^+(i)$ is not a path and $s\geq 5$.

Let $G$ be a $k$-connected triangle-free graph with $\delta(G)\geq k+\max(P_s^+(i))+[P_s^+(i)\text{ is bad}]$. Let $U$ be an end of $G$ if $\kappa(G)=k$, and let $U=G$ otherwise, such that $U$ contains no subgraph $H\cong P_s^+(i)$ with $\kappa(G-V(H))\geq k$.

Let $c$ be the vertex of degree three in $P_s^+(i)$, and let $L_j=v_j^1\cdots v_j^{l_j}$ for $j\in[3]$ be the three paths of $P_s^+(i)-c$, where
$v_j^1$ is adjacent to $c$ and $v_j^0=c$. Without loss of generality, we assume that $l_3=1$. If $P_s^+(i)$ is bad, choose $L_1$ to be a longest leg; otherwise, choose $L_1$ such that $l_1$ is odd and $l_1\geq3$. Thus $l_1\geq2$ and $l_1$ is even when $P_s^+(i)$ is bad.

By the choice of $L_1$, we have $\max(P_s^+(i))+[P_s^+(i)\text{ is bad}] \geq \max(P_s^+(i)-v_1^{l_1}) +[P_s^+(i)-v_1^{l_1}\text{ is bad}]$. Hence, by the minimality of $s$, $U$ contains a subgraph $T\cong P_s^+(i)-v_1^{l_1}$ such that $\kappa(G-V(T))\geq k$. Let $\phi$ be the embedding from $P_s^+(i)-v_1^{l_1}$ to $T$, let $\phi(c)=c'$ and $\phi(v_j^r)=u_j^r$ for every vertex $v_j^r\in V(P_s^+(i)-v_1^{l_1})$. Write $L_1'=u_1^1\cdots u_1^{l_1-1}$, $L_2'=u_2^1\cdots u_2^{l_2}$, and $L_3'=u_3^1$. Also, let $u_1^0=c'$.

\begin{claim}\label{degreek+1}
Let $T^\star$ be a tree in $G$ isomorphic to $P_s^+(i)-V(L_1[v_1^j,v_1^{l_1}])$, where $j\in[l_1]$. Then $|N_G(u)\cap V(G-T^\star)|\ge k+\left \lceil \frac{l_1-j+1}{2}\right \rceil$ for every vertex $u\in V(G)$.
\end{claim}
\begin{proof}
Choose an arbitrary vertex $u\in V(G)$, then $|N_G(u)\cap V(T^\star)|\leq \max(T^\star)$ by Lemma \ref{vT^*}. Note that $\max(T^\star)=\left \lceil \frac{j-1}{2}\right \rceil+\left \lceil \frac{l_2}{2}\right \rceil+1$. Therefore,
\begin{align*}
|N_G(u)\cap V(G-V(T^\star))|&\geq \delta(G)-|N_G(u)\cap V(T^\star)|\\
&\geq k+\max(P_s^+(i))+[P_s^+(i)\text{ is bad}]-\max(T^\star)\\
&=k+(\left \lceil \frac{l_1}{2}\right \rceil+\left \lceil \frac{l_2}{2}\right \rceil+1)+[P_s^+(i)\text{ is bad}]-(\left \lceil \frac{j-1}{2}  \right \rceil+\left \lceil \frac{l_2}{2}\right \rceil+1)\\
&\geq k+\left \lceil \frac{l_1-j+1}{2}  \right \rceil.
\end{align*}
For the last inequality, if $P_s^+(i)$ is good, then $l_1$ is odd and $\left \lceil \frac{l_1}{2}\right \rceil-\left \lceil \frac{j-1}{2}  \right \rceil=\left \lceil \frac{l_1-j+1}{2}  \right \rceil$. Otherwise, $l_1$ is even and the additional term $[P_s^+(i)\text{ is bad}]=1$ yields the required inequality. This completes the proof of Claim \ref{degreek+1}.
\end{proof}
\begin{claim}\label{degreel_1}
For any $j\in [l_1]$, if $U$ contains a subgraph $T^\star\cong P_s^+(i)-V(L_1[v_1^j, v_1^{l_1}])$ with an isomorphism $\sigma: P_s^+(i)-V(L_1[v_1^j,v_1^{l_1}])\to T^\star$ such that $\kappa(G-V(T^\star))=k$ and $N_G(\sigma(v_1^{j-1}))\cap V(F_j)\neq\emptyset$, where $F_j$ is an end of $G-V(T^\star)$ with $F_j\subseteq U$, then $U$ contains a subgraph $H \cong P_s^+(i)$ such that $\kappa(G-V(H))\ge k$.
\end{claim}
\begin{proof}
Choose $w\in N_G(\sigma(v_1^{j-1}))\cap V(F_j)$. We distinguish two cases according to the parity of $l_1-j+1$.

If $l_1-j+1$ is odd, then $\delta(G-V(T^\star))\geq k+\left \lceil \frac{l_1-j+1}{2} \right \rceil$ by Claim \ref{degreek+1}, and thus there exists a path $P_w$ of order $l_1-j+1$ starting from $w$ in $F_j$ such that $\kappa(G-(V(T^\star)\cup V(P_w)))\geq k$ by Theorem \ref{F1}. Joining $P_w$ to $T^\star$ by the edge $\sigma(v_1^{j-1})w$ gives the required subgraph.

If $l_1-j+1$ is even, since $\delta(G-V(T^\star))\geq k+\left \lceil \frac{l_1-j+1}{2} \right \rceil\geq k+1$, $\kappa(G-(V(T^\star)\cup \{w\}))\geq k$ by Theorem \ref{F1}. Moreover, $|F_j|\geq\delta(G-V(T^*))-k+1\geq2$. By Lemma \ref{subend}, $\kappa(G-(V(T^\star)\cup \{w\}))=k$ and $G-(V(T^\star)\cup \{w\})$ contains an end $F'\subseteq F_j$ such that $N_G(w)\cap V(F')\neq \emptyset$. Let $T'$ be obtained from $T^\star$ by adding $w$ and the edge $\sigma(v_1^{j-1})w$. Then $T'\subseteq U$ is isomorphic to $P_s^+(i)-V(L_1[v_1^{j+1},v_1^{l_1}])$ with $v_1^j$ corresponding to $w$. Applying Claim \ref{degreek+1} to $T'$ with the index $j+1$, we obtain $\delta(G-V(T'))\ge k+\left\lceil\frac{l_1-j}{2}\right\rceil$. Since $l_1-j$ is odd, the preceding odd case applies.
\end{proof}

Let $G_1=G-V(T)$. We claim that $\kappa(G-V(T))=k$. If not, then $\kappa(G-V(T))>k$ since $\kappa(G-V(T))\ge k$. By Claim \ref{degreek+1}, we obtain that $|N_G(u_1^{l_1-1})\cap V(G-V(T))|\geq k+1>0$. Thus, $u_1^{l_1-1}$ has at least $k+1-k\geq 1$ neighbors in $U$. Hence, there exists a vertex $u'\in N_{U}(u_1^{l_1-1})\cap V(G-V(T))$. Clearly, $G[V(T)\cup \{u'\}]$ contains a subgraph isomorphic to $P_s^+(i)$ in $U$ and $\kappa(G-(V(T)\cup \{u'\}))\geq k$ since $\kappa(G-V(T))>k$, a contradiction.

Choose an end $F\subseteq U$ of $G_1$ such that $N_G(T)\cap V(F)\neq\emptyset$. When $\kappa(G)=k$, choose an edge $xy$ of $U$.
Since $U$ is triangle-free, the neighborhoods of $x$ and $y$ in $U$ are disjoint. Hence $|U|\ge d_{U}(x)+d_{U}(y)\ge2(k+\max(P_s^+(i))+[P_s^+(i)\text{ is bad}]-k)\ge s>|T|$, thus such an end exists by Lemma \ref{subend}. When $\kappa(G)>k$, any end of $G_1$ has the required property; otherwise its neighbor set in $G_1$ would be a $k$-cutset of $G$.

By Claim \ref{degreek+1} with $j=l_1$, every vertex of $T$ has at least $k+1$ neighbors in $G_1$. Consequently, for every $X\subseteq V(T)$ with $N_G(X)\cap V(F)=\emptyset$, Lemma \ref{lem:add-vertex-outside-end} implies that $G[V(G_1)\cup X]$ has connectivity $k$ and still has $F$ as an end. We proceed by distinguishing the following three cases according to the adjacency between $T$ and $F$:

{\bf Case A.} $N_G(c'L_1')\cap V(F)\neq \emptyset$.

Let $j$ be the largest index such that $N_G(u_1^j)\cap V(F)\neq\emptyset$, where $0\leq j\leq l_1-1$. Let $X=\{u_r^1:j+1\leq r\leq l_1-1\}$ and $T'=T-X$. By the choice of $j$, $N_G(X)\cap V(F)=\emptyset$. Thus $\kappa(G-V(T'))=k$ and $F$ remains an end of $G-V(T')$ by Lemma \ref{lem:add-vertex-outside-end}. Applying Claim \ref{degreel_1} with $j+1$ gives a contradiction.

{\bf Case B.} $N_G(L_3')\cap V(F)\neq \emptyset$.

By Case A, $N_G(c'L_1')\cap V(F)=\emptyset$. Let $T'=T-\{u_r^1:2\leq r\leq l_1-1\}$. Then $\kappa(G-V(T'))=k$ and $F$ remains an end of $G-V(T')$ by Lemma \ref{lem:add-vertex-outside-end}. Interchanging the roles of $u_1^1$ and $u_1^3$ in $T'$, Claim \ref{degreel_1} with $j=2$ gives a contradiction.

{\bf Case C.} $N_G(V(L_2'))\cap V(F)\neq \emptyset$.

Let $j$ be the largest index such that $N_G(u_j^2)\cap V(F)\neq\emptyset$. Let $Q$ be the subpath of $L_2'[u_j^2,u_1^2]c'L_1'$ of order $\min\{j+l_1,l_2+2\}$ starting at $u_j^2$. Since $j\leq l_2$ and $l_1\geq2$, the other endvertex of $Q$ lies in $L_1'$. By Cases A and B and the choice of $j$, exactly one endvertex of $Q$ has a neighbor in $F$ and $N_G(T-V(Q))\cap V(F)=\emptyset$. By Lemma \ref{lem:add-vertex-outside-end}, $\kappa(G-V(Q))=k$ and $F$ remains an end of $G-V(Q)$.

We next obtain such a path of order $l_2+2$. For every path $R$ of order at most $l_2+2$ in $G$, $\delta(G-V(R))
\geq (k+\left \lceil \frac{l_1}{2} \right \rceil +\left \lceil \frac{l_2}{2} \right \rceil +1)-\left \lceil \frac{l_2+2}{2} \right \rceil\geq k+1$.
Suppose that $|Q|<l_2+2$, and let $a$ be the endvertex of $Q$ having a neighbor $x$ in $F$. By Theorem \ref{F1}, $\kappa(G-(V(Q)\cup\{x\}))\geq k$.
Since $|F|\geq\delta(G-V(Q))-k+1\geq2$, it follows from Lemma \ref{subend} that $\kappa(G-(V(Q)\cup\{x\}))=k$ and $G-(V(Q)\cup\{x\})$ contains an end
$F'\subseteq F$ such that $N_G(x)\cap V(F')\neq\emptyset$. Extending $Q$ by the edge $ax$, we obtain a longer path with exactly one endvertex having a neighbor in $F'$. Repeating this argument yields a path of order $l_2+2$ with the same properties.

Choose a path $P=x_1\cdots x_{l_2+2}\subseteq U$ and an end $F_P\subseteq U$ of $G-V(P)$ such that $\kappa(G-V(P))=k$ and exactly one endvertex of $P$ has a neighbor in $F_P$ with $|F_P|$ as small as possible. Without loss of generality, assume that $N_G(x_1)\cap V(F_P)=\emptyset$ and $N_G(x_{l_2+2})\cap V(F_P)\neq\emptyset$.

We have $N_G(x_2)\cap V(F_P)=\emptyset$. Otherwise, viewing $P$ as a copy of $P_s^+(i)-V(L_1)$ with $c$ corresponding to $x_2$, Claim \ref{degreel_1} with $j=1$ gives a contradiction.

Choose $y\in N_G(x_{l_2+2})\cap V(F_P)$. Since $\delta(G-V(P))\geq k+\left \lceil \frac{l_1}{2} \right \rceil\geq k+1$, $\kappa(G-(V(P)\cup \{y\}))\geq k$ by Theorem \ref{F1}. Moreover, since $|F_P|\geq\delta(G-V(P))-k+1\geq 2$, $\kappa(G-(V(P)\cup \{y\}))=k$ and $G-(V(P)\cup \{y\})$ contains an end
$F'\subseteq F_P$ such that $N_G(y)\cap V(F')\neq\emptyset$ by Lemma \ref{subend}.

Let $P'=x_2\cdots x_{l_2+2}y$.  The vertex $x_1$ has no neighbor in $F'$ and has at least $\delta(G)-\left \lceil \frac{l_2+2}{2} \right \rceil-1\geq k+\left \lceil \frac{l_1}{2} \right \rceil-1\geq k$ neighbors in $G-(V(P)\cup \{y\})$. By Lemma \ref{lem:add-vertex-outside-end}, adjoining $x_1$ implies $\kappa(G-V(P'))=k$ and $F'$ is still an end of $G-V(P')$. Since $N_G(x_2)\cap V(F_P)=\emptyset$, exactly one endvertex of $P'$ has a neighbor in $F'$. Together with $P'\subseteq U$ and $F'\subseteq U$, this contradicts the choice of $(P,F_P)$ as $|F'|<|F_P|$.

This completes the proof of Theorem \ref{triangle}.
\end{proof}
\section{Proof of Proposition \ref{propk=2}}\label{Prop1.1}
Suppose to the contrary that no such edge exists. Let $E_0=\{e\in E(G):e\text{ is contained in a }\linebreak2\text{-cutset of }G\}$, and let $E_1=E(G)\setminus E_0$. Since $G-V(e)$ is not $2$-connected for every edge $e\in E(G)$, every edge in $E_1$ is contained in a $3$-cutset of $G$.

For $i=0,1$, let $\mathcal S_i$ be the family of $(2+i)$-cutsets $S$ of $G$ such that $G[S]$ contains an edge in $E_i$. Choose $S\in \mathcal S_0\cup \mathcal S_1$ and a semifragment $F$ of $G$ to $S$ such that $|F|$ is as small as possible. Let $\bar{F}=G-(V(F)\cup S)$.
\begin{claim}\label{F^*}
For any $S^*\in \mathcal S_0\cup \mathcal S_1$ and any semifragment $F^*$ of $G$ to $S^*$, $|F^*|\geq 3$.
\end{claim}
\begin{proof}
Since $S^*\in \mathcal S_0\cup \mathcal S_1$, $G[S^*]$ contains an edge, say $ab$. Since $G$ is triangle-free, every vertex of $F^*$ is adjacent to at most one of $a$ and $b$. If $F^*$ has no edge, then every vertex of $F^*$ has degree at most $|S^*|-1\leq 2$, contradicting $\delta(G)\geq 3$. Hence $G[F^*]$ contains an edge $uv$. Thus $|F^*|\geq |N_G(u)\cup N_G(v)|-|S^*|=d_G(u)+d_G(v)-|S^*|\geq 6-3=3$.
\end{proof}
The minimality of $F$ implies that $F$ is connected. Take an edge $uv\in E(F)$. If $uv\in E_i$, choose $S'\in\mathcal S_i$ containing $u$ and $v$. Let $F'$ be a semifragment of $G$ to $S'$, and set $\bar{F'}=G-(S'\cup V(F'))$.  We use the following notation for the nine intersections:
\begin{align*}
A_1=F\cap F', A_2=S\cap F',A_3=\bar F\cap F',\\
A_4=F\cap S',A_5=S\cap S',A_6=\bar F\cap S',\\
A_7=F\cap \bar{F'},A_8=S\cap \bar{F'},A_9=\bar F\cap \bar{F'}.
\end{align*}
\begin{claim}\label{Ssize}
$S\in \mathcal S_1$ and $E(G[F])\subseteq E_1$.
\end{claim}
\begin{proof}
Suppose that $S\in \mathcal S_0$. Then $|S|=2$. Since $G[S]$ is an edge, by interchanging $F'$ and $\bar{F'}$ if necessary, we may assume that $S\subseteq A_2\cup A_5$. Thus $A_8=\emptyset$.

If $A_7\neq\emptyset$, then $A_4\cup A_5\cup A_8$ is a cut set of $G$ with $|A_4\cup A_5\cup A_8|=|A_4\cup A_5|\leq 3$. Since $uv\in E(G[A_4\cup A_5])$, $A_4\cup A_5\cup A_8\in \mathcal S_0\cup \mathcal S_1$ and $A_7$ is a semifragment of $G$ to $A_4\cup A_5\cup A_8$ with $|A_7|<|F|$, contradicting the minimality of $F$.

If $A_9\neq\emptyset$, then $A_5\cup A_6\cup A_8$ is a cut set of $G$ and thus $|A_5\cup A_6\cup A_8|\geq 2$ since $G$ is $2$-connected. Recall that $\{u,v\}\subseteq A_4$ and $S\subseteq A_2\cup A_5$. So we have $6\le4+2\le|A_2\cup A_4\cup A_5|+|A_5\cup A_6\cup A_8|=|S|+|S'|\le5$, a contradiction.

We have now obtained $A_7=A_8=A_9=\emptyset$ and thus $\bar{F'}=\emptyset$, contradicting the definition of a semifragment. Therefore, $S\in \mathcal S_1$.

Now suppose that $uv\in E_0$. Then $S'=\{u,v\}$, and hence $A_4=\{u,v\}$ and $A_5=A_6=\emptyset$. Choose $xy\in E(G[S])\cap E_1$. By interchanging $F'$ and $\bar{F'}$ if necessary, we may assume that $\{x,y\}\subseteq A_2$. Thus $|A_8|\leq 1$.

If $A_9\neq\emptyset$, then $A_8$ is a cutset of order at most one, a contradiction. Hence $A_9=\emptyset$. Since $\bar{F}\neq\emptyset$, we have $A_3\neq\emptyset$. Then $A_2$ is a cut set containing $xy\in E_1$, so $|A_2|\geq 3$. Consequently, $A_2=S$ and $A_8=\emptyset$. It follows that $\bar{F'}=A_7$ is a nonempty semifragment to $S'$ with $A_7\subsetneq F$, contradicting the minimality of $|F|$. Since $uv$ was arbitrary, $E(G[F])\subseteq E_1$.
\end{proof}
By Claim \ref{Ssize}, $|S|=|S'|=3$. Choose $xy\in E(G[S])\cap E_1$. Let $S=\{x,y,z\}.$ By interchanging $F'$ and $\bar{F'}$ if necessary, we may assume that $\{x,y\}\subseteq A_2\cup A_5.$
\begin{claim}
$A_7=\emptyset$.
\end{claim}
\begin{proof}
If not, then $A_7$ is a semifragment of $G$ to $A_4\cup A_5\cup A_8$. Since $A_7\subsetneq F$, the minimality of $|F|$ implies that $A_4\cup A_5\cup A_8\notin \mathcal S_0\cup \mathcal S_1$. Moreover, since $uv\in E_1$ and $\{u,v\}\subseteq A_4$, this forces $|A_4\cup A_5\cup A_8|\geq 4$ and thus $|A_4|>|A_2|$ since $|A_2\cup A_5\cup A_8|=|S|=3$.
Since $|S|,|S'|\leq3$, $|A_2\cup A_5\cup A_6|\leq |S|+|S'|-|A_4\cup A_5\cup A_8|\leq 2$.

If $A_3\neq\emptyset$, then $A_2\cup A_5\cup A_6$ is a cutset of order at most $2$ containing the edge $xy\in E_1$, a contradiction to the definition of $E_1$. Hence $A_3=\emptyset$. It follows that $|F'|=|A_1|+|A_2|+|A_3|<|A_1|+|A_4|+|A_7|=|F|$. This contradicts the choice of $F$.
\end{proof}
\begin{claim}\label{A425}
$|A_4|=|A_2\cup A_5|=2$.
\end{claim}
\begin{proof}
Suppose either $|A_4|>2$ or $|A_2\cup A_5|>2$. Then $|A_2\cup A_4\cup A_5|\geq 5$. If $A_9\neq \emptyset$, then $A_5\cup A_6\cup A_8$ is a cut set of $G$ and thus $|A_5\cup A_6\cup A_8|\geq 2$. This gives $|S|+|S'|>6$, a contradiction. Thus $A_9=\emptyset$. Now $A_7=A_9=\emptyset$, by Claim \ref{F^*} we have $|A_8|\geq 3$. Hence $|S|=|A_2\cup A_5\cup A_8|\geq 5$, a contradiction.
\end{proof}
By Claim \ref{A425}, we have $A_4=\{u,v\}$ and $A_2\cup A_5=\{x,y\}$. This implies $A_8=\{z\}$. Moreover, since $A_7=\emptyset$, every neighbor of $z$ in $F$ lies in $A_4$. Since $G$ is triangle-free, $z$ is nonadjacent to at least one of $u$ and $v$. Without loss of generality, assume that $zu\notin E(G)$. Then $\{v,x,y\}$ is a $3$-cutset of $G$ with $xy\in E(G[\{v,x,y\}])$. Moreover, $A_1\cup\{u\}$ is a semifragment of $G$ to $\{v,x,y\}$, and $|A_1\cup\{u\}|<|F|$. This contradicts the minimality of $|F|$.
\hfill$\qed$

\end{spacing}

\end{document}